\documentclass[12pt, reqno]{amsart}
     \makeatletter
     \def\section{\@startsection{section}{1}%
     \z@{.7\linespacing\@plus\linespacing}{.5\linespacing}%
     {\bfseries
     \centering
     }}
     \def\@secnumfont{\bfseries}
     \makeatother
\newtheorem{theorem}{Theorem}[section]
\newtheorem{lemma}[theorem]{Lemma}
\newtheorem{proposition}[theorem]{Proposition}
\newtheorem{corollary}[theorem]{Corollary}
\theoremstyle{definition}

\theoremstyle{remark}
\newtheorem{remark}[theorem]{Remark}
\numberwithin{equation}{section}
\def \btheta {\mbox{\boldmath $\theta $}}

\def \bvarphi {\mbox{\boldmath $\varphi$}}

\def \diag  {\text {\rm diag}}

\def \diag  {\text {\rm diag}}
\def \lim   {\text {\rm lim}}

\begin{document}

\title[Two Remarks on Normality Preserving]{Two
Remarks on Normality Preserving Borel Automorphisms of
$\mathbb{R}^n$}

{\author{K. R. Parthasarathy}
 \address{Theoretical Statistics and Mathematics
 Unit, Indian Statistical Institute, 7, S. J. S. Sansanwal Marg, New
 Delhi-110 016, India}
 \email{krp@isid.ac.in}}


%

\begin{abstract}
Let  $T$ be a bijective map on $\mathbb{R}^n$ such that both $T$ and
$T^{-1}$ are Borel measurable. For any $\btheta \in \mathbb{R}^n$ and
any real $n \times n$ positive definite matrix $\Sigma,$ let $N
(\btheta, \Sigma)$ denote the $n$-variate normal (gaussian)
probability
measure on $\mathbb{R}^n$ with mean vector $\btheta$ and
covariance matrix $\Sigma.$ Here we prove the following two results:
(1) Suppose $N(\btheta_j, I)T^{-1}$ is gaussian for $0 \leq j \leq n$
where $I$ is the identity matrix and $\{\btheta_j - \btheta_0, 1 \leq
j \leq n \}$ is a basis for $\mathbb{R}^n.$ Then $T$ is an affine
linear transformation; (2) Let $\Sigma_j = I + \varepsilon_j
\mathbf{u}_j \mathbf{u}_j^{\prime},$ $1 \leq j \leq n$ where
$\varepsilon_j > -1$ for every $j$ and $\{\mathbf{u}_j, 1 \leq j \leq
n \}$ is a basis of unit vectors in $\mathbb{R}^n$ with
$\mathbf{u}_j^{\prime}$ denoting the transpose of the column vector
$\mathbf{u}_j.$ Suppose $N(\mathbf{0}, I)T^{-1}$ and $N (\mathbf{0},
\Sigma_j)T^{-1},$ $1 \leq j \leq n$ are gaussian. Then $T(\mathbf{x})
=
\sum\limits_{\mathbf{s}} 1_{E_{\mathbf{s}}} V \mathbf{s} U
\mathbf{x}$ a.e. $\mathbf{x}$ where $\mathbf{s}$ runs over the set of
$2^n$ diagonal matrices of order $n$ with diagonal entries $\pm 1,$
$U,\, V$ are $n \times n$ orthogonal matrices and $\{
E_{\mathbf{s}}\}$ is a collection of $2^n$ Borel subsets of
$\mathbb{R}^n$ such that $\{
E_{\mathbf{s}}\}$ and $\{V \mathbf{s} U (E_{\mathbf{s}})\}$ are
partitions of $\mathbb{R}^n$ modulo Lebesgue-null sets and for every
$j,$ $V \mathbf{s} U \Sigma_j (V \mathbf{s} U)^{-1}$ is independent of
all $\mathbf{s}$ for which the Lebesgue measure of $E_{\mathbf{s}}$
is positive. The converse of this result also holds.
\vskip0.1in
Our results constitute a sharpening of the results of S. Nabeya and
T. Kariya \cite{nskt}.
\vskip0.1in
\noindent{\it AMS Subject Classification\,\,\,}
62E10
\end{abstract}

\maketitle

\newpage
\section{Introduction and preliminaries}
\setcounter{equation}{0}

Let $(\Omega_i, \mathcal{F}_i),$ $i=1,2$ be measurable spaces and let
$\mathcal{P}_i$ be a set of probability measures on $\mathcal{F}_i$
for each $i.$ In the context of statistical theory, Basu and Khatri
\cite{bdkcg} posed the important problem of characterizing the set of
all measurable maps $T: \Omega_1 \rightarrow \Omega_2$ satisfying the
property that $\mu T^{-1} \in \mathcal{P}_2$ whenever $\mu \in
\mathcal{P}_1.$ They present many examples and solve the problem
completely when $\Omega_1 =  \mathbb{R}^n,$ $\Omega_2 =
\mathbb{R}^m,$ $\mathcal{F}_1$ and $\mathcal{F}_2$ are Borel
$\sigma$-algebras and $\mathcal{P}_1$ and $\mathcal{P}_2$ are sets of
all gaussian laws. When $\Omega_1 = \Omega_2 = \mathbb{R}^m,$
$$\mathcal{P}_1 = \mathcal{P}_2 =  \{\mu A^{-1}, A \,\, \mbox{an
affine linear transformation of} \,\, \mathbb{R}^n \}$$ 
where $\mu$ is
a fixed probability measure the problem looks particularly
interesting. When $n=1$ and $\mu$ is the symmetric Cauchy law this
problem has been given a complete solution by Letac \cite{lg}. In the
gaussian case on the real line there are many variations due to Ghosh
\cite{ghk} and Mase \cite{ms}.
\vskip0.1in
Nabeya and Kariya \cite{nskt} have studied the problem when $T$ is a
{\it Borel automorphism} of $\mathbb{R}^n,$ i.e., a bijective map on
$\mathbb{R}^n$ such that $T$ and $T^{-1}$ are Borel measurable,
$\mathcal{P}_2$ is the class of all nonsingular gaussian measures on
$\mathbb{R}^n$ and $\mathcal{P}_1$ is a restricted class of such
gaussian measures. Here we follow their approach and obtain a
sharpened version of their results after a more detailed analysis in
two cases when $\mathcal{P}_2$ is the set of all nonsingular gaussian
measures in $\mathbb{R}^n$ but $\mathcal{P}_1$ is (1) a set of $n+1$
gaussian measures with a fixed nonsingular covariance matrix and mean
vectors $\btheta_j,$ $0 \leq j \leq n$ where
$\{\btheta_j -\btheta_0, 0 \leq j \leq n  \}$ is a
basis for $\mathbb{R}^n;$ (2) a set of $(n+1)$ gaussian measures with
a fixed mean vector $\btheta$ and nonsingular covariance
matrices $\Sigma_j,$ $0 \leq j \leq n$ such that
Rank$(\Sigma_j - \Sigma_0)=1$ and there exists a basis
$\{\mathbf{u}_j, 1 \leq j \leq n \}$ of unit vectors in
$\mathbb{R}^n$ satisfying the condition $\mathbf{u}_j \in \,\,{\rm
Range}\,\, (\Sigma_j - \Sigma_0)$ for each $j = 1,2,
\ldots, n.$ The main result of Nabeya and Kariya in the gaussian case
follows as a corollary.
\vskip0.1in
We conclude this section with some notations and definitions that
will be used in the following sections. Let $GL(n),$ $\mathcal{S}_{+}(n)$ and
$D_n$ denote respectively the set of all real $n \times n$
nonsingular matrices, positive definite matrices and diagonal
matrices with entries $\pm 1.$ For any matrix $A$ denote by
$A^{\prime}$ its transpose. We express any element of $\mathbb{R}^n$
as a $n \times 1$ matrix and for any $\btheta \in
\mathbb{R}^n,$ $\Sigma \in \mathcal{S}_{+} (n),$ denote by $N(\btheta,
\Sigma)$ the gaussian probability measure in $\mathbb{R}^n$ with
density function $(2 \pi)^{-n/2} |\Sigma|^{-1/2} \exp - 1/2
(\mathbf{x} - \btheta)^{\prime} \Sigma^{-1} (\mathbf{x} -
\btheta), |\Sigma|$ being the determinant of $\Sigma.$ A Borel
map $T$ on $\mathbb{R}^n$ is called an {\it affine automorphism} if
$T(\mathbf{x}) = A \mathbf{x} + \mathbf{a}$ a.e. $\mathbf{x}$ with
respect to Lebesgue measure where $A \in GL(n)$ and $\mathbf{a} \in
\mathbb{R}^n.$ All almost everywhere statements in $\mathbb{R}^n$
will be with respect to Lebesgue measure. By a {\it Lebesgue
partition} over $D_n$ we mean a collection $\{E_{\mathbf{s}},
\mathbf{s} \in D_n \}$ of Borel sets with the property that all the
sets $\mathbb{R}^n \backslash \underset{\mathbf{s} \in D_{n}}{\cup}
E_{\mathbf{s}},$ $E_{\mathbf{s}} \cap E_{\mathbf{t}},$ $\mathbf{s}
\neq \mathbf{t}$ have zero Lebesgue measure.
\vskip0.1in
If $T$ is an affine automorphism of $\mathbb{R}^n$ such that $T
(\mathbf{x}) = A \mathbf{x} + \mathbf{a}$ a.e. $\mathbf{x}$ for some
$A \in GL(n),$ $\mathbf{a} \in \mathbb{R}^n$ and $N
(\btheta, \Sigma)$ is a gaussian probability measure with mean
vector $\btheta$ and covariance matrix $\Sigma$ then
$N(\btheta, \Sigma)T^{-1} = N(A \btheta+\mathbf{a}, A
\Sigma A^{\prime}).$  We say that two Borel automorphisms $T_1$ and
$T_2$ are {\it affine equivalent} if $T_2 = R T_1 S$ a.e. for some
affine automorphisms $R$ and $S.$ If $T$ is a Borel automorphism of
$\mathbb{R}^n$ and $N(\btheta, \Sigma)T^{-1} = N
(\bvarphi, \Psi)$ for some $\btheta,$
$\bvarphi \in \mathbb{R}^n,$ $\Sigma,$ $\Psi \in GL(n)$ then
there exists a Borel automorphism $\widetilde{T}$ affine equivalent
to $T$ such that $N(\mathbf{0}, I)\widetilde{T}^{-1} = N(\mathbf{0},
I),$ i.e., $\widetilde{T}$ preserves the standard gaussian
probability measure with mean vector $\mathbf{0}$ and covariance
matrix $I.$ Thus problems of the type we have described in the
context of gaussian measures can be translated to the case when the
Borel automorphism preserves the standard gaussian measure in
$\mathbb{R}^n.$

\section{Borel automorphisms preserving the normality of a pair of
normal distributions}
\setcounter{equation}{0}
Let $T$ be a Borel automorphism of $\mathbb{R}^n$ such that
\begin{equation}
N(\btheta_i, \Sigma_i)T^{-1} = N(\bvarphi_i,
\Psi_i), \,\,i=1,2   \label{eq2.1}
\end{equation}
where $\Sigma_i, \Psi_i \in \mathcal{S}_{+}(n),$ $\btheta_i,
\bvarphi_{i} \in \mathbb{R}^n$ for each $i.$  Our aim is to
establish the following proposition which, together with its proof,
is a slight variation of Lemma 2.1 and Lemma 2.2 together in the
paper of Nabeya and Kariya \cite{nskt}.
\vskip0.1in
\begin{proposition}\label{prop2.1}
Under condition \eqref{eq2.1} the following hold:
\begin{enumerate}
 \item[(1)] $\Psi_1^{1/2}$  $\Psi_2^{-1}$  $\Psi_1^{1/2}$  and
$\Sigma_1^{1/2}$  $\Sigma_2^{-1}$  $\Sigma_1^{1/2}$ have the same
characteristic polynomial.
\vskip0.1in
\item[(2)] 
          $ (\bvarphi_1 - \bvarphi_2)^{\prime}
\Psi_2^{-1} (\Psi_1^{-1} - z \Psi_2^{-1})^{-1} \Psi_2^{-1}
(\bvarphi_1 - \bvarphi_2) $\\
$= (\btheta_1 - \btheta_2)^{\prime}
\Sigma_2^{-1} (\Sigma_1^{-1} - z \Sigma_2^{-1})^{-1} \Sigma_2^{-1}
(\btheta_1 - \btheta_2) $
as analytic functions of $z.$
\vskip0.1in
\item[(3)] $(\bvarphi_1 - \bvarphi_2)^{\prime}
\Psi_2^{-1} (\bvarphi_1 - \bvarphi_2) =
(\btheta_1 - \btheta_2)^{\prime} \Sigma_2^{-1}
(\btheta_1 - \btheta_2).$

\vskip0.1in
\item[(4)] $(T(\mathbf{x}) - \bvarphi_2)^{\prime} \Psi_2^{-1}
(T(\mathbf{x}) - \bvarphi_2) - (T(\mathbf{x}) -
\bvarphi_1)^{\prime} \Psi_1^{-1} (T(\mathbf{x}) -
\bvarphi_1)$\\
$=(\mathbf{x} - \btheta_2)^{\prime} \Sigma_2^{-1} (\mathbf{x} -
\btheta_2) - (\mathbf{x} -
\btheta_1)^{\prime} \Sigma_1^{-1} (\mathbf{x} -
\btheta_1)  $\\
a.e. $\mathbf{x}.$ 
\end{enumerate}
\end{proposition}

\begin{proof}
Let $f_i,$ $\widetilde{f}_i$ denote respectively the density
functions of $N(\btheta_i, \Sigma_i),$ $N (\bvarphi_i,
\Psi_i),$ $i =1,2.$ Equation \eqref{eq2.1} implies, in particular,
that the Lebesgue measure $L$ in $\mathbb{R}^n$ is equivalent to the
measure $LT^{-1}$ and for any real $t$

\begin{equation}
\left [\frac{\widetilde{f}_1}{\widetilde{f}_2} (T(\mathbf{x})) \right
]^{2t} = \left [\frac{f_1}{f_2} (\mathbf{x}) \right ]^{2t} \quad
\mbox{a.e.} \,\,\mathbf{x}. \label{eq2.2}
\end{equation}
Integrating both sides with respect to the probability measure
$N(\btheta_1, \Sigma_1)$ and using \eqref{eq2.1} for $i=1$ we
get for all $t$ in a neighbourhood of $0$
\begin{eqnarray*}
\lefteqn{|\Psi_2|^t |\Psi_1|^{-(t+1/2)} \left |(t+\frac{1}{2})
\Psi_1^{-1} - t \Psi_2^{-1} \right |^{-\frac{1}{2}}}\\
&\times& \exp t (\bvarphi_1 - \bvarphi_2)^{\prime}
\Psi_2^{-1} (\bvarphi_1 - \bvarphi_2) + \\
&&  t^2 \left
\{(\bvarphi_1 - \bvarphi_2)^{\prime} \Psi_2^{-1} \left
((t+\frac{1}{2}) \Psi_1^{-1} - t \Psi_2^{-1} \right )^{-1} \Psi_2^{-1}
(\bvarphi_1 - \bvarphi_2) \right \} \\
&=& |\Sigma_2|^t |\Sigma_1|^{-(t+\frac{1}{2})} |(t+\frac{1}{2})
\Sigma_1^{-1} - t \Sigma_2^{-1}|^{-\frac{1}{2}}\\
&\times& \exp \,\,t \, (\btheta_1 -
\btheta_2)^{\prime} \Sigma_2^{-1} (\btheta_1 -
\btheta_2) \\
&& + t^2 \left \{(\btheta_1 - \btheta_2)^{\prime}
\Sigma_2^{-1} \left ((t+\frac{1}{2}) \Sigma_1^{-1} - t \Sigma_2^{-1}
\right
)^{-1} \Sigma_2^{-1} (\btheta_1 - \btheta_2 )  \right
\}.
\end{eqnarray*}
Squaring both sides and rearranging the terms this can be expressed as
$$\frac{\left |(t+\frac{1}{2}) \Psi_1^{-1} - t \Psi_2^{-1}  \right
|}{\left |(t+\frac{1}{2}) \Sigma_1^{-1} - t \Sigma_2^{-1} \right |} =
e^{\chi(t)}$$
where the left hand side and the function $\chi(t)$ in the exponent
of the right hand side are rational functions of $t$ in a
neighbourhood of $0$ in $\mathbb{R}.$ By exactly the same
arguments using the theory of analytic functions as in the proof of
Lemma 2.2 in \cite{nskt} we now conclude
\begin{eqnarray}
|\Sigma_1| \,\, \left |(t+\frac{1}{2}) \Sigma_1^{-1} -t \Sigma_2^{-1} 
\right |  &=& |\Psi_1| \,\, \left |(t+\frac{1}{2}) \Psi_1^{-1} -t
\Psi_2^{-1}  \right |, \nonumber \\ 
\chi(t) &=& 0   \label{eq2.3}
\end{eqnarray}
for all $t$ in a neighbourhood of $0.$ The first part of Equation
\eqref{eq2.3} implies
property (1) of the proposition and, in particular,
\begin{equation}
|\Sigma_1| \,\,|\Sigma_2|^{-1} = |\Psi_1| \,\,|\Psi_2|^{-1}.  
\label{eq2.5}
\end{equation}
Now the second part of equation \eqref{eq2.3} can be written
as
\begin{eqnarray*}
\lefteqn{t \left \{(\btheta_1 - \btheta_2)^{\prime}
\Sigma_2^{-1} (\btheta_1 - \btheta_2)-
(\bvarphi_1 - \bvarphi_2)^{\prime}
\Psi_2^{-1} (\bvarphi_1 - \bvarphi_2)\right \} } \\
&+& t^2 \left \{ (\btheta_1 - \btheta_2)^{\prime}
\Sigma_2^{-1} \left ((t+\frac{1}{2}) \Sigma_1^{-1} - t \Sigma_2^{-1} 
\right
)^{-1} \Sigma_2^{-1} (\btheta_1 - \btheta_2) \right
. \\
&& \left . - (\bvarphi_1 - \bvarphi_2)^{\prime}
\Psi_2^{-1} \left ((t+\frac{1}{2}) \Psi_1^{-1} - t \Psi_2^{-1} \right
)^{-1} \Psi_2^{-1} (\bvarphi_1 - \bvarphi_2) \right
\} = 0
\end{eqnarray*}
in a neighbourhood of $0.$ Dividing by $t$ and letting $t \rightarrow
0$ we get property (3). After deleting the first term, dividing by
$t^2$ and replacing the parameter $(t+\frac{1}{2})^{-1} t$ by $z$ we
get property (2). Going back to equation \eqref{eq2.2}, putting
$t=1,$ using \eqref{eq2.5} and taking logarithms we get property (4).
This completes the proof.
\end{proof}
\vskip0.1in
\begin{corollary}\label{cor2.2}
Under condition \eqref{eq2.1} the following hold:
\begin{enumerate}
 \item[(i)] If $\Sigma_1 = \Sigma_2$ then $\Psi_1 = \Psi_2;$
\vskip0.1in
 \item[(ii)] If $\btheta_1 = \btheta_2$ then
$\bvarphi_1 =\bvarphi_2. $
\end{enumerate}
\end{corollary}

\vskip0.1in
\begin{proof}
Property (i) follows from property (1) in Proposition \ref{prop2.1}
whereas property (ii) is a consequence of property (3) of the same
proposition.
\end{proof}

 \section{The first characterization theorem}
 \setcounter{equation}{0}

Here we consider a Borel automorphism $T$ of $\mathbb{R}^n$
satisfying the property that $N(\btheta, \Sigma) T^{-1}$  is a
gaussian probability measure for a fixed $\Sigma$ in
$\mathcal{S}_{+}(n)$ and the
mean vector $\btheta$ varying in a natural set of $(n+1)$
points in $\mathbb{R}^n.$ We shall make use of the remark at the end
of Section 1.
\vskip0.1in
\begin{theorem}\label{thm3.1}
Let $T$ be a Borel automorphism of $\mathbb{R}^n$ and let
$\{\btheta_j, 0 \leq j \leq n \} \subset \mathbb{R}^n$  be a
set of $(n+1)$ points with the property that $\btheta_j -
\btheta_0,$ $j=1,2,\ldots,n$ is a basis for $\mathbb{R}^n.$
Suppose
$$N (\btheta_j, I) T^{-1} = N (\bvarphi_j, \Psi_j),
\,\, 0 \leq j \leq n. $$ 
Then $T$ is an affine automorphism.
\end{theorem}
\vskip0.1in
\begin{proof}
 By part (i) of Corollary \ref{cor2.2} there exists $\Psi \in
\mathcal{S}_{+}(n)$ such that $\Psi_j = \Psi$ for every $j.$ From
property (3)
of Proposition \ref{prop2.1} we have
\begin{equation}
(\bvarphi_i - \bvarphi_j)^{\prime} \Psi^{-1}
(\bvarphi_i - \bvarphi_j) = (\btheta_i -
\btheta_j)^{\prime} (\btheta_i -
\btheta_j) \label{eq3.1}
\end{equation}
for all $i,j \in \{0,1,2,\ldots,n \}.$ Hence the correspondence
$$ (\btheta_i -\btheta_0) \rightarrow \Psi^{-1/2}
(\bvarphi_i -\bvarphi_0), \,\, 1 \leq i \leq n$$
is distance preserving in $\mathbb{R}^n$ with the Euclidean metric.
Since $\{\btheta_i - \btheta_0, \,\, 1 \leq i \leq
n\}$ is a basis for $\mathbb{R}^n$ it follows that there exists an
orthogonal matrix $U$ such that
\begin{equation}
U (\btheta_i - \btheta_0)  = \Psi^{-1/2}
(\bvarphi_i - \bvarphi_0), \,\, 1 \leq i \leq n.
\label{eq3.2}
\end{equation}
Using property (4) of Proposition \ref{prop2.1} we have
\begin{eqnarray*}
\lefteqn{(T (\mathbf{x}) - \bvarphi_i)^{\prime} \Psi^{-1}
\left ( T (\mathbf{x}) - \bvarphi_i \right ) - (T
(\mathbf{x}) - \bvarphi_j)^{\prime} \Psi^{-1} (T
(\mathbf{x}) - \bvarphi_j) }\\
&=& (\mathbf{x} - \btheta_i)^{\prime} (\mathbf{x} -
\btheta_i)-(\mathbf{x} - \btheta_j)^{\prime}
(\mathbf{x} -
\btheta_j), i,j \in \{0,1,2, \ldots, n\} \\
&& \mbox{a.e.} \,\,\mathbf{x}.
\end{eqnarray*}
Expressing $T (\mathbf{x}) - \bvarphi_i = T(\mathbf{x}) -
\bvarphi_0 + \bvarphi_0 - \bvarphi_i,$
$\mathbf{x} - \btheta_i = \mathbf{x} - \btheta_0 +
\btheta_0 - \btheta_i$ for each $i$ and expanding
both sides of the equation above we obtain by using \eqref{eq3.1} the
equation
$$(\bvarphi_i - \bvarphi_j)^{\prime} \Psi^{-1}
(T(\mathbf{x}) - \bvarphi_0) = (\btheta_i -
\btheta_j)^{\prime} (\mathbf{x} - \btheta_0)
\,\,\mbox{a.e.} \,\, \mathbf{x}. $$
Writing
$$ A = \left [\begin{array}{c}(\btheta_1 -
\btheta_0)^{\prime} \\ \vdots \\ (\btheta_n -
\btheta_0)^{\prime}\end{array} \right ], \quad B=\left
[\begin{array}{c}(\bvarphi_1 -
\bvarphi_0)^{\prime} \Psi^{-1}\\ \vdots \\ (\bvarphi_n
-
\bvarphi_0)^{\prime} \Psi^{-1}\end{array} \right ]$$
and using \eqref{eq3.2} we conclude that $A$ and $B$ are nonsingular
matrices and
$$T (\mathbf{x}) = \bvarphi_0 + B^{-1} A (\mathbf{x} -
\btheta_0) \,\,\mbox{a.e.}$$
In other words $T$ is an affine automorphism.
\end{proof}

\vskip0.1in
\begin{corollary}\label{cor3.2}
 Let $T$ be a Borel automorphism of $\mathbb{R}^n$ and let
$\{\btheta_j, 0 \leq j \leq n\} \subset \mathbb{R}^n$ be  a
set of $n+1$ points such that $\{\btheta_j -
\btheta_0, 1 \leq j \leq n\} \subset \mathbb{R}^n$  is a
basis for $\mathbb{R}^n.$ Suppose there exists an $n \times n$
nonsingular covariance matrix $\Sigma$ such that
$$N (\btheta_j, \Sigma) T^{-1} = N(\bvarphi_j,\,\,
\Psi_j),
\,\, 0 \leq j \leq n.$$
Then $T$ is an affine automorphism.
\end{corollary}

\vskip0.1in
\begin{proof}
Define $T^{\prime} (\mathbf{x}) = T (\Sigma^{1/2} \mathbf{x}).$ Then
$T^{\prime}$ is a Borel automorphism for which
$$N (\Sigma^{-1/2} \btheta_j, I) T^{\prime^{-1}} =
N(\bvarphi_j, \Psi_j), \, 0 \leq j \leq n. $$
Since $\{\Sigma^{-1/2} (\btheta_j - \btheta_0), 1
\leq j \leq n\}$ is a basis for $\mathbb{R}^n$ it follows that
$T^{\prime}$ and therefore $T$ is an affine automorphism.
\end{proof}

\vskip0.1in
\begin{remark}
 Corollary \ref{cor3.2} is false if $T$ is just a Borel measurable
map. Using a result of Linnik (Theorem 4.3.1 in  \cite{lyv}),
Mase \cite{ms} has shown that given any finite set of normal
distributions $\{N_j, 1 \leq j \leq k \}$ on $\mathbb{R}$ there
exists a Borel map $T: \mathbb{R} \rightarrow \mathbb{R}$ such that
$N_j T^{-1}$ is the probability measure of the standard normal
distribution on $\mathbb{R}.$
\end{remark}

\section{The second characterization theorem}
 \setcounter{equation}{0}

Here we consider the case of a Borel automorphism $T$ of
$\mathbb{R}^n$ such that $N(\btheta, \Sigma)T^{-1}$ is a
gaussian probability measure for a fixed $\btheta$ in
$\mathbb{R}^n$ but $\Sigma$ varying in a set of $n + 1$ covariance
matrices with $n$ of them being rank one perturbations of the
remaining one. Without loss of generality one may assume the fixed
mean vector $\btheta$ to be $\mathbf{0}.$ We begin with a
lemma.
\vskip0.1in
\begin{lemma}\label{lem4.1} 
Let $T$ be a Borel automorphism of $\mathbb{R}^n$ such that
\begin{eqnarray*}
N(\mathbf{0}, I) T^{-1} &=& N (\mathbf{0}, I), \\
N(\mathbf{0}, \Sigma) T^{-1} &=& N (\mathbf{\eta}, \Psi)
\end{eqnarray*}
where $\Sigma = I + \varepsilon \mathbf{u}, \mathbf{u}^{\prime},$
$\mathbf{u}$ is a unit vector in $\mathbb{R}$ and $\varepsilon$ is a
real nonzero scalar such that $\varepsilon > -1.$ Then $\mathbf{\eta}
= \mathbf{0}$ and $\Psi = I + \varepsilon \mathbf{v}
\mathbf{v}^{\prime}$ for some unit vector $\mathbf{v}.$ Furthermore
\begin{equation}
(\mathbf{v}^{\prime} T(\mathbf{x}))^2 =  (\mathbf{u}^{\prime}
\mathbf{x})^2 \quad \mbox{a.e.} \,\, \mathbf{x}.\label{eq4.1}
\end{equation}
\end{lemma}

\vskip0.1in
\begin{proof}
By property (ii) of Corollary \ref{cor2.2}, $\mathbf{\eta} =
\mathbf{0}.$ By property (1) of Proposition \ref{prop2.1}, $\Sigma $
and $\Psi$ have the same characteristic polynomial and therefore
$\Psi = I + \varepsilon \mathbf{v} \mathbf{v}^{\prime}$ for some unit
vector  $\mathbf{v}.$ By property (4) of Proposition \ref{prop2.1} we
have
\begin{eqnarray}
\lefteqn{ T(\mathbf{x})^{\prime} (I + \varepsilon \mathbf{v}  
\mathbf{v}^{\prime})^{-1} T ( \mathbf{x}  ) - T ( \mathbf{x} 
)^{\prime}  T ( \mathbf{x}  )  } \nonumber \\ 
&=& \mathbf{x}^{\prime} (I + \varepsilon \mathbf{u}
\mathbf{u}^{\prime} )^{-1} \mathbf{x} - \mathbf{x}^{\prime} \mathbf{x}
\,\,\, \mbox{a.e.} \,\, \mathbf{x}. \label{eq4.2}
\end{eqnarray}
Since, for any unit vector $\mathbf{w},$ $(I + \varepsilon \mathbf{w}
\mathbf{w}^{\prime} )^{-1} = I - \frac{\varepsilon}{1+\varepsilon}
\mathbf{w} \mathbf{w}^{\prime},$ \eqref{eq4.1} follows from
\eqref{eq4.2}.
\end{proof}
\vskip0.1in
\begin{lemma}\label{lem4.2}
Let $T$ be a Borel automorphism of $\mathbb{R}^n$ such that
\begin{enumerate}
 \item[(i)] $N (\mathbf{0}, I)T^{-1} = N(\mathbf{0}, I),$
\vskip0.1in
\item[(ii)] $N(\mathbf{0}, \Sigma_j) T^{-1} = N(\bvarphi_j,
\Psi_j),$ $1 \leq j \leq n$
\end{enumerate}
where $\Sigma_j = I + \varepsilon_j \mathbf{u}_j
\mathbf{u}_j^{\prime},
$ $ 1 \leq j \leq n,$ $ 0 \neq \varepsilon_j > -1$ $\forall$ $j$ and
$\{\mathbf{u}_j, 1 \leq j \leq n \}$ is a basis of unit vectors in
$\mathbb{R}^n.$ Then there exist $A,B$ in $GL(n)$ and a Lebesgue
partition $\{E_{\mathbf{s}}, \mathbf{s} \in D_n\}$ of $\mathbb{R}^n$
over $D_n$ such that
\begin{enumerate}
 \item[(i)] $\{B \mathbf{s} A E_{\mathbf{s}}, \mathbf{s} \in D_n \}$
is a Lebesgue partition over $D_n,$
\vskip0.1in
\item[(ii)] $T (\mathbf{x}) = \Sigma_{\mathbf{s} \in D_{n}}
1_{E_{\mathbf{s}}} (\mathbf{x}) B \,\,\mathbf{s}\,\, A \mathbf{x}$
a.e.
$\mathbf{x},$
\vskip0.1in
\item[(iii)] $T^{-1} (\mathbf{x}) = \Sigma_{\mathbf{s} \in
D_{n}} 1_{B \mathbf{s} A (E_{\mathbf{s}})} (\mathbf{x}) A^{-1}
\mathbf{s} B^{-1} \mathbf{x}$ a.e. $\mathbf{x}.$
\end{enumerate}
\end{lemma}
\vskip0.1in
\begin{proof}
 From Lemma \ref{lem4.1} we have $\bvarphi_j = \mathbf{0}$
$\forall$ $j$ and
$$\Psi_j = I + \varepsilon_j \mathbf{v}_j \mathbf{v}_j^{\prime},
\quad 1 \leq j \leq n$$
where each $\mathbf{v}_j$ is a unit vector. Furthermore
$$(\mathbf{v}_j^{\prime} T(\mathbf{x}))^2 = (\mathbf{u}_j^{\prime}
\mathbf{x})^2 \quad \mbox{a.e.} \,\,\mathbf{x},\,\, 1 \leq j \leq n.
$$
Denoting $\mathbf{s} = \diag (s_1, s_2, \ldots, s_n)$ with $s_j = \pm
1$ and `$\diag$' indicating diagonal matrix define the Borel sets
$$E_{\mathbf{s}} = \left \{\mathbf{x} \left | \mathbf{v}_j^{\prime}
T(\mathbf{x}) = s_j \mathbf{u}_j^{\prime} \mathbf{x} \quad \forall
\,\, j = 1,2,\ldots,n \right . \right \}, \quad \mathbf{s} \in D_n. $$
Then it follows that for the Lebesgue measure $L,$
$$L \left (\mathbb{R}^n \backslash \underset{\mathbf{s} \in
D_{n}}{\cup} E_{\mathbf{s}} \right ) = 0. $$
Write
$$A = \left [\begin{array}{c}\mathbf{u}_1^{\prime} \\
\mathbf{u}_2^{\prime} \\ \vdots \\ \mathbf{u}_n^{\prime} \end{array}
\right ], \quad C = \left [\begin{array}{c}\mathbf{v}_1^{\prime} \\
\mathbf{v}_2^{\prime} \\ \vdots \\ \mathbf{v}_n^{\prime} \end{array}
\right ].$$
Then
$$C T(\mathbf{x}) = \mathbf{s} A \mathbf{x} \quad \forall \,\,
\mathbf{x} \in E_{\mathbf{s}}. $$
Choose $\mathbf{s} \in D_n$ such that $L (E_{\mathbf{s}}) > 0.$ Then
we know that $\{\mathbf{s} A \mathbf{x}, \mathbf{x} \in E_{\mathbf{s}}
\}$ spans $\mathbb{R}^n$ and therefore $\{C T(\mathbf{x}), \mathbf{x}
\in E_{\mathbf{s}}\}$ spans $\mathbb{R}^n.$ Thus $C \in GL(n).$
Writing $B=C^{-1}$ we have
$$T(\mathbf{x}) B \mathbf{s} A \mathbf{x} \quad \forall
\,\,\mathbf{x} \in E_{\mathbf{s}}.$$
Let now $\mathbf{s},$ $\mathbf{t} \in D_n,$ $\mathbf{s} \neq t.$ If
$\mathbf{x} \in E_{\mathbf{s}} \cap E_{\mathbf{t}} $ we have
$\mathbf{s} A \mathbf{x} = \mathbf{t} A \mathbf{x}.$ In other words,
$A\mathbf{x}$ is an eigenvector of $\mathbf{s} \mathbf{t}$ and
$\mathbf{s} \mathbf{t} \neq I.$ Hence $L (E_{\mathbf{s}} \cap
E_{\mathbf{t}}) = 0.$ Thus $\{E_{\mathbf{s}}, \mathbf{s} \in D_n \}$
is a Lebesgue partition over $D_n$ and
$$T(\mathbf{x}) = \Sigma_{\mathbf{s}} 1_{E_{\mathbf{s}}} (\mathbf{x}) B
\mathbf{s} A \mathbf{x} \quad \mbox{a.e.} \,\, \mathbf{x}.$$
Putting $T(\mathbf{x}) = \mathbf{y}$ and solving for $\mathbf{x}$
from this equation we get
$$T^{-1} (\mathbf{y}) = \Sigma_{\mathbf{s}} 1_{B \mathbf{s} A
(E_{\mathbf{s}})} (\mathbf{y}) A^{-1} \mathbf{s} B^{-1} \mathbf{y}
\quad \mbox{a.e.} \,\,\mathbf{y} $$
where $\{B \mathbf{s} A (E_{\mathbf{s}}), s \in D_n \}$ is also a
Lebesgue partition. This proves all the required properties (i)-(iii)
of the lemma.
\end{proof}
\vskip0.1in
\begin{lemma}\label{lem4.3}
 In Lemma 4.2, the matrices $A$ and $B$ can be chosen to be
orthogonal.
\end{lemma}
\vskip0.1in
\begin{proof}
Following the notations of Lemma \ref{lem4.2} and its proof denote
$F_{\mathbf{s}} = B \mathbf{s} A (E_{\mathbf{s}})$ and by $(\Sigma,
\Psi)$ any of the pairs $(I,I),$ $(\Sigma_j, \Psi_j),$ $1 \leq j \leq
n.$ Then we have $|\Sigma| = |\Psi|$ and $N (\mathbf{0}, \Sigma)T^{-1} =
N(\mathbf{0}, \Psi).$ Thus, for any $\lambda \in \mathbb{R}^n$ we have
\begin{eqnarray*}
\lefteqn{\int e^{\mathbf{\lambda}^{\prime} \mathbf{x} - \frac{1}{2}
\mathbf{x}^{\prime} \Psi^{\prime} \mathbf{x} } d\mathbf{x} } \\
&=& \int e^{\mathbf{\lambda}^{\prime} T(\mathbf{x})  - \frac{1}{2}
\mathbf{x}^{\prime} \Sigma^{-1} \mathbf{x}   } d \mathbf{x} \\ 
&=& \Sigma_{\mathbf{s}} \int_{E_{\mathbf{s}}}
e^{\mathbf{\lambda}^{\prime} B \mathbf{s} A \mathbf{x} - \frac{1}{2}
\mathbf{x}^{\prime} \Sigma^{-1} \mathbf{x} }\\
&=& |\,\,|AB|^{-1}\,\,| \Sigma_{\mathbf{s}} \int_{F_{\mathbf{s}}}
e^{\mathbf{\lambda}^{\prime} \mathbf{x} - \frac{1}{2}
\mathbf{x}^{\prime} (A^{-1} \mathbf{s} B^{-1})^{\prime} \Sigma^{-1}
A^{-1} \mathbf{s} B^{-1} \mathbf{x}  } d \mathbf{x}.
\end{eqnarray*}
By the uniqueness of Laplace transforms we conclude that
$$e^{-\frac{1}{2} \mathbf{x}^{\prime} \Psi^{-1} \mathbf{x}} =
|\,\,|AB|^{-1}\,\,|  \Sigma_{\mathbf{s} \in D_n}
1_{F_{\mathbf{s}}}(\mathbf{x})
e^{-\frac{1}{2} \mathbf{x}^{\prime} (A^{-1} \mathbf{s}
B^{-1})^{\prime} \Sigma^{-1} A^{-1} \mathbf{s} B^{-1} \mathbf{x}    }
\quad \mbox{a.e.}\,\, \mathbf{x}.$$
Suppose $L(E_{\mathbf{s}}) > 0.$ Then $L(F_{\mathbf{s}}) > 0.$ Since
$\{F_t\}$ is a Lebesgue partition we have
$$\log |\,|AB|\,| = \frac{1}{2} \mathbf{x}^{\prime} \left
(\Psi^{-1} -
(A^{-1} \mathbf{s} B^{-1})^{\prime} \Sigma^{-1} A^{-1} \mathbf{s}
B^{-1} \right ) \mathbf{x} \quad \mbox{a.e.} \,\, \mathbf{x}
\in F_{\mathbf{s}}. $$
In other words the quadratic form on the right hand side is a
constant on a set of positive Lebesgue measure. A simple argument
based on Fubini's theorem implies that $|\,|AB|\,| = 1$ and
$$B \mathbf{s} A \Sigma (B \mathbf{s} A)^{\prime} = \Psi \quad
\mbox{whenever} \quad L(E_{\mathbf{s}}) > 0. $$
Choosing $(\Sigma, \,\,\Psi) = (I,I)$ we conclude that $B \mathbf{s}
A$ is
orthogonal. Let $B=VK,$ $A=HU$ where $U,V$ are orthogonal and $K,H$
are positive definite. Then $K \mathbf{s} H^2 \mathbf{s} K = I$ and
therefore $K = \mathbf{s} H^{-1} \mathbf{s}$ and $B \mathbf{s} A = V
\mathbf{s} U.$ This shows that
$$V \mathbf{s} U \Sigma (V \mathbf{s} U)^{\prime} = \Psi \quad
\mbox{whenever}\quad L(E_{\mathbf{s}})>0 $$
and Lemma \ref{lem4.2} holds with $A,B$ replaced by $U,V$
respectively.
\end{proof}
\vskip0.1in
\begin{theorem}\label{thm4.4}
Let $T$ be a Borel automorphism of $\mathbb{R}^n$ and let $\Sigma_j = I
+ \varepsilon_j \mathbf{u}_j \mathbf{u}_j^{\prime},$ $1 \leq j \leq
n$ where $\varepsilon_j > -1$ for every $j$ and $\{ \mathbf{u}_j, 1
\leq j \leq n\}$ is a basis of unit vectors in $\mathbb{R}^n.$ Then
$N(\mathbf{0}, I)T^{-1} = N(\mathbf{0},I)$ and $N(\mathbf{0},
\Sigma_j)T^{-1}$ is a gaussian probability measure for every $j$ if and
only if there exist $n \times n$ orthogonal matrices $U,V$ and a
Lebesgue partition $\{E_{\mathbf{s}}, \mathbf{s} \in D_n  \}$ of
$\mathbb{R}^n$ such that the following hold:
\begin{enumerate}
 \item[(i)] $T(\mathbf{x}) = \Sigma_{\mathbf{s} \in D_{n}}
1_{E_{\mathbf{s}}} (\mathbf{x}) V \mathbf{s} U \mathbf{x}$ a.e.
$\mathbf{x};$
\vskip0.1in
  \item[(ii)] $\{ V \mathbf{s} U (E_{\mathbf{s}}), \mathbf{s} \in
D_n\}$ is also a Lebesgue partition. In such a case there exists a
basis $\{\mathbf{v}_j, 1 \leq j \leq n\}$ of unit vectors in
$\mathbb{R}^n$ such that
\vskip0.1in
 \item[(iii)] $ V \mathbf{s} U \mathbf{u}_j = \mathbf{v}_j $ if $L
(E_{\mathbf{s}}) > 0,$ $\forall$ $1 \leq j \leq n;$
\vskip0.1in
  \item[(iv)] $N(\mathbf{0}, \Sigma_j)T^{-1} = N(\mathbf{0}, \Psi_j)
\quad \forall \,\,j$ where $\Psi_j = I + \varepsilon_j \mathbf{v}_j
\mathbf{v}_j^{\prime},$ $1 \leq j \leq n.$
\end{enumerate}
\end{theorem}
\vskip0.1in
\begin{proof}
The only if part is contained in Lemma \ref{lem4.2} and Lemma
\ref{lem4.3} and their proofs. Conversely, suppose there exist
orthogonal matrices $U,V$ and a Lebesgue partition $\{E_{\mathbf{s}},
\mathbf{s} \in D_n\}$ such that (i), (ii) and (iii) hold. Then for
every pair $(\Sigma, \Psi) =(I,I), (\Sigma_j, \Psi_j), 1 \leq j \leq n $
we have 
\begin{eqnarray*}
\lefteqn{\int e^{\mathbf{\lambda}^{\prime} T(\mathbf{x}) -
\frac{1}{2} \mathbf{x}^{\prime} \Sigma^{-1}  \mathbf{x}} d \mathbf{x}   
} \\
&=& \Sigma_{\mathbf{s}} \int_{E_{\mathbf{s}}} e^{
\mathbf{\lambda}^{\prime} V \mathbf{s} U \mathbf{x} -
\frac{1}{2}\mathbf{x}^{\prime} \Sigma^{-1}  \mathbf{x}  } d\mathbf{x}\\
&=& \Sigma_{\mathbf{s}} \int_{V \mathbf{s} U(E_{\mathbf{s}})}
e^{\mathbf{\lambda}^{\prime} \mathbf{x} - \frac{1}{2}
\mathbf{x}^{\prime} \Psi^{-1} \mathbf{x} }d \mathbf{x} \\
&=& \int e^{\mathbf{\lambda}^{\prime} \mathbf{x} - \frac{1}{2}
\mathbf{x}^{\prime} \Psi^{-1} \mathbf{x} }d \mathbf{x}.
\end{eqnarray*}
Since $|\Sigma| = |\Psi|$ it follows that $N(\mathbf{0}, \Sigma) T^{-1} =
N(\mathbf{0}, \Psi).$
\end{proof}
\vskip0.1in
\begin{corollary}\label{cor4.5}
Let $T$ be a Borel automorphism of $\mathbb{R}^n$ and let $\Sigma_j,$
$0 \leq j \leq n$ be elements in $\mathcal{S}_{+}(n)$ satisfying the
following
conditions:
\begin{enumerate}
 \item[(1)] rank$(\Sigma_j - \Sigma_0) = 1.$
\vskip0.1in
  \item[(2)] Range$(\Sigma_j - \Sigma_0) = \mathbb{R} \mathbf{u}_j$ with
$\mathbf{u}_j$ as a unit vector where $\mathbf{u}_1, \mathbf{u}_2,
\ldots, \mathbf{u}_n$ is a basis for $\mathbb{R}^n.$
\end{enumerate}
Then $N(\mathbf{0}, \Sigma_j)T^{-1} =N(\mathbf{0}, \Psi_j)$ for every
$j,$ for some covariance matrices $\Psi_0,\Psi_1, \ldots, \Psi_n $ if
and only if there exist orthogonal matrices $U,V$ of order $n$ and a
Lebesgue partition $\{E_{\mathbf{s}}, \mathbf{s} \in D_n\}$ such that
$\{V \mathbf{s} U (E_{\mathbf{s}} ), \mathbf{s} \in D_n \}$ is also a
Lebesgue partition and
$$T(\mathbf{x}) = \Sigma_{\mathbf{s} \in D_{n}} 1_{E_{\mathbf{s}}}
(\Sigma_0^{-1/2} \mathbf{x}) \Psi_0^{1/2} V s U \Sigma_0^{-1/2}
\mathbf{x} \quad \mbox{a.e.} \,\,\mathbf{x}. $$
\end{corollary}

\vskip0.1in
\begin{proof}
This is immediate from Theorem \ref{thm4.4} and the fact that the
Borel automorphism $T^{\prime}$ defined by
$$T^{\prime} (\mathbf{x}) = \Psi_0^{-1/2} T(\Sigma_0^{-1/2} \mathbf{x})
$$
preserves $N(0, I)$ and transforms $N(\mathbf{0}, \Sigma_0^{-1/2}
\Sigma_j \Sigma_0^{-1/2})$ to a mean zero gaussian probability measure
for each $j = 1,2, \ldots, n.$
\end{proof}

\vskip0.1in
\begin{remark}\label{rem4.6}
Let $\mathbb{R} = F_{j+} \cup F_{j-}$ be a partition of $\mathbb{R}$
into two disjoint symmetric Borel subsets, i.e. $F_{j+} = -F_{j+}$
and $F_{j-} = -F_{j-}$ for each $j=1,2,\ldots,n$ and for $\mathbf{s}
= \diag (s_1, s_2, \ldots, s_n)$ in $D_n$ let
$$F_{\mathbf{s}} = F_{s_{1}}  \times F_{s_{2}}  \times \cdots \times
F_{s_{n}}$$
where
$$F_{s_{j}} = \left \{\begin{array}{lcl} F_{j+}  &
\mbox{if}& s_j = +1,\\  F_{j-} &\mbox{if}& s_j = -1. \end{array}
\right . $$
If $U,V$ are orthogonal matrices of order $n,$ putting
$E_{\mathbf{s}} = U^{-1} F_{\mathbf{s}}$ for every $\mathbf{s}$ we
note that $\{E_{\mathbf{s}}, \mathbf{s} \in D_n \}$ and $\{V
\mathbf{s} U E_{\mathbf{s}}, \mathbf{s} \in D_n\}$ are Lebesgue
partitions of $\mathbb{R}^n.$ Indeed, $V \mathbf{s} U E_{\mathbf{s}}
= V F_{\mathbf{s}}$ for every $\mathbf{s}.$
\end{remark}

\vskip0.1in
\begin{remark}\label{rem4.7}
If $\{E_{\mathbf{s}}, \mathbf{s} \in D_n \}$ happens to be a trivial
Lebesgue partition in the sense that for some $\mathbf{s}_0,$ $L
(\mathbb{R}^n \backslash E_{\mathbf{S}_{0}}) = 0$ then the Borel
automorphism $T$ in Theorem \ref{thm4.4} is an orthogonal
transformation a.e. whereas in Theorem 4.5 is a nonsingular linear
transformation a.e. in $\mathbb{R}^n.$
\end{remark}

\vskip0.1in
\begin{remark}\label{rem4.8}
In Theorem \ref{thm4.4}, let
$$D = \{ \mathbf{s} \left |L(E_{\mathbf{s}}) > 0  \right . \} \subset
D_n $$
and let $T$ satisfy (i)-(iv). Suppose
$$\mathcal{C} = \{A \left | A \in \mathcal{S}_{+}(n), \right
.\mathbf{s}
UAU^{-1} \mathbf{s}\,\,\mbox{is independent of} \,\,\mathbf{s}
\,\,\mbox{varying in} \,\, D.\}. $$
Then
$$N(\mathbf{0}, A)T^{-1} = N(\mathbf{0}, V \mathbf{s} UA U^{-1}
\mathbf{s} V^{-1}) \,\, \forall \,\, A \in \mathcal{C}, \mathbf{s}
\in D.$$
\end{remark}
\vskip0.2in
\section*{Acknowledgement} The author is grateful to B. G.
Manjunath for bringing his attention to all the references in this
paper and particularly to the results of S. Nabeya and T. Kariya
\cite{nskt}.


\begin{thebibliography}{99}

\bibitem{bdkcg} Basu, D. and Khatri C. G, {\it On some
characterizations of statistics,} Sankhy$\bar{{\rm a}}$ Ser.A {\bf
31} (1969) 199-208.

\bibitem{ghk} Ghosh J. K., {\it Only linear transformations preserve
normality,} Sankhy$\bar{{\rm a}}$ Ser.A {\bf 31} (1969) 309-312.

\bibitem{lg} Letac G. {\it Which functions preserve Cauchy laws?} 
Proc. Amer. Math. Soc. {\bf 67} (1977) 277-286.

\bibitem{lyv} Linnik Yu. V., {\it Statistical Problems with Nuisance
Parameters} (American Mathematical Society) (1968).

\bibitem{ms} Mase S. {\it Some theorems on normality preserving
transformations, } Sankhy$\bar{{\rm a}}$ Ser. A {\bf 39} (1977)
186-190.

\bibitem{nskt} Nabeya S. and Kariya T. {\it Transformations
preserving normality and Wishart-ness,} J. Multivariate Anal. {\bf
20} (1986) 251-264.

\end{thebibliography}
\end{document}